\documentclass[12pt,reqno]{amsart}
\usepackage{amsmath, amsthm, amssymb, mathtools}
\usepackage{hyperref}

\usepackage{mathtools}

\advance \topmargin by -\headheight
\advance \topmargin by -\headsep
     
\evensidemargin \oddsidemargin
\newtheorem{thm}{Theorem}[section]
\newtheorem{lemma}[thm]{Lemma}

\newtheorem{cor}[thm]{Corollary}

\theoremstyle{definition}
\newtheorem{defn}[thm]{Definition}

\theoremstyle{remark}

\numberwithin{equation}{section}

\newcommand*\wrapletters[1]{\wr@pletters#1\@nil}
\def\wr@pletters#1#2\@nil{#1\allowbreak\if&#2&\else\wr@pletters#2\@nil\fi}

\def\alp{{\alpha}} 
\def\bet{{\beta}}  
\def\gam{{\gamma}}

\def\lam{{\lambda}}

\def\eps{\varepsilon}

\def\le{\leqslant} \def\ge{\geqslant}

\def\d{{\,{\rm d}}}

\def \bN {\mathbb N}

\def \bQ {\mathbb Q}
\def \bR {\mathbb R}
\def \bZ {\mathbb Z}

\def \bb {\mathbf b}
\def \bc {\mathbf c}

\def \br {\mathbf r}

\def \bu {\mathbf u}
\def \bv {\mathbf v}
\def \bx {\mathbf x}

\def \by {\mathbf y}

\def \bzero {\mathbf 0}

\def \bmu {{\boldsymbol{\mu}}}
\def \bbeta {\boldsymbol{\beta}}
\def \bgam {\boldsymbol{\gam}}
\def \bxi {{\boldsymbol{\xi}}}

\def \fm {\mathfrak m}
\def \fn {\mathfrak n}

\def \fr {\mathfrak r}

\def \ft {\mathfrak t}

\def \fM {\mathfrak M}
\def \fN {\mathfrak N}

\def \fR {\mathfrak R}

\def \fU {\mathfrak U}

\def \cI {\mathcal I}

\def \cR {\mathcal R}

\def \cV {\mathcal V}

\def \hZ {\hat{Z}}

\def \tf {\tilde f}

\def \sinc {\mathrm{sinc}}
\def \meas {\mathrm{meas}}

\begin{document}
\title[Sums of cubes with shifts]{Sums of cubes with shifts}
\author[Sam Chow]{Sam Chow}
\address{School of Mathematics, University of Bristol, University Walk, Clifton, Bristol BS8 1TW, United Kingdom}
\email{Sam.Chow@bristol.ac.uk}
\subjclass[2010]{11D75, 11E76, 11P05}
\keywords{Diophantine inequalities, forms in many variables, inhomogeneous polynomials}
\thanks{}
\date{}
\begin{abstract} 
Let $\mu_1, \ldots, \mu_s$ be real numbers, with $\mu_1$ irrational. We investigate sums of shifted cubes $F(x_1,\ldots,x_s) = (x_1 - \mu_1)^3 + \ldots + (x_s - \mu_s)^3$. We show that if $\eta$ is real, $\tau >0$ is sufficiently large, and $s \ge 9$, then there exist integers $x_1 > \mu_1, \ldots, x_s > \mu_s$ such that $|F(\bx)- \tau| < \eta$. This is a real analogue to Waring's problem. We then prove a full density result of the same flavour for $s \ge 5$. For $s \ge 11$, we provide an asymptotic formula. If $s \ge 6$ then $F(\bZ^s)$ is dense on the reals. Given nine variables, we can generalise this to sums of univariate cubic polynomials.
\end{abstract}
\maketitle

\section{Introduction}
\label{intro}

Research on diophantine inequalities in many variables has hitherto focussed predominantly on inequalities of the shape
\begin{equation} \label{shape}
|\lam_1 x_1^k + \ldots + \lam_s x_s^k| < \eta.
\end{equation}
It is hoped that understanding such inequalities will provide prophetic insights into general diophantine inequalities. In this sense, inequalities of the form \eqref{shape} play a r\^ole analogous to that played by Waring's problem in the context of diophantine equations in many variables. We propound a new analogue to Waring's problem. Let $s$ be a positive integer, and let $\mu_1, \ldots, \mu_s$ be real numbers, with $\mu_1$ irrational. We investigate the values taken by sums of shifted cubes 
\[ F(x_1,\ldots,x_s) = (x_1 - \mu_1)^3 + \ldots + (x_s - \mu_s)^3 \]
for integers $x_i > \mu_i$ ($1 \le i \le s$). Let $\eta > 0$ be a real number.

\begin{thm} \label{Waring} 
Let $s \ge 9$, and let $\tau$ be a sufficiently large positive real number. Then there exist integers $x_1 > \mu_1, \ldots, x_s > \mu_s$ such that
\begin{equation} \label{main} 
|F(\bx) - \tau| < \eta.
\end{equation}
\end{thm}

If one is only interested in showing that $F(\bZ^s)$ is dense on the reals, then six variables suffice.

\begin{thm} \label{six} 
Let $s \ge 6$. Then $F(\bZ^s)$ is dense on $\bR$.
\end{thm}

A heuristic application of the Davenport-Heilbronn circle method suggests that the conclusion of Theorem \ref{Waring} is valid whenever $s \ge 4$. Combining our ideas with those of Parsell and Wooley \cite{PW2013}, we establish a full density result for $s \ge 5$. For real numbers $A < B$, let $Z(A,B)$ denote the set of $\tau \in [A,B]$ such that \eqref{main} has no solution $\bx \in \bZ^s$ such that $x_i > \mu_i$ for all $i$. Let $N$ be a large positive real number, and put $Z(N) = Z(0,N)$.

\begin{thm} \label{FullDensity} 
Let $s \ge 5$. Then
\[ \meas(Z(N)) = o(N). \]
\end{thm}

Our proof is easily adapted to show that for all real numbers $\tau \in [-N,N]$, save for those lying in a set of measure $o(N)$, there exists $\bx \in \bZ^s$ satisfying \eqref{main}. The result of Theorem \ref{FullDensity} cannot be obtained for $s =3$. 

\begin{thm} \label{LowerDensity}
Assume that $s=3$ and $\eta < 1/4$. Then
\[ \meas(Z(N)) > N/2. \]
\end{thm}

Given eleven variables we can obtain an asymptotic formula for the number of `positive' solutions to \eqref{main}. When $\tau > 0$ is large, denote by $N(\tau) = N_{s,\eta,\boldsymbol{\mu}}(\tau)$ the number of integral solutions $\bx \in (\mu_1,\infty)\times \ldots \times (\mu_s,\infty)$ to \eqref{main}. 

\begin{thm} \label{AsymptoticFormula}
Let $s \ge 11$. Then
\[ N(\tau) \sim  2 \eta \Gamma(4/3)^s \Gamma(s/3)^{-1}\tau^{s/3-1}. \]
\end{thm}

By a simplification of our methods, we may obtain a similar asymptotic formula for sums of five shifted squares. We can also handle sums of nine univariate cubic polynomials, subject to an irrationality condition, improving on the thirteen variable result apparent from Freeman's work \cite{Fre2003}. We take the following definition from \cite{Fre2003}.

\begin{defn} \label{irr} Let $k \ge 2$ be an integer. For $i=1,2,\ldots,s$, let $h_i(x)$ be a degree $k$ polynomial with real coefficients given by
\[ h_i(x) = \beta_{ik}x^k+ \ldots + \beta_{i1}x + \beta_{i0}. \]
The polynomials $h_1, \ldots, h_s$ satisfy the \emph{irrationality condition} if there exist $i_1, i_2 \in \{1,2,\ldots,s \}$ and $j_1,j_2 \in \{1,2,\ldots,k\}$ such that $\beta_{i_2j_2} \ne 0$ and $\beta_{i_1j_1} / \beta_{i_2j_2}$ is irrational.
\end{defn}

\begin{thm} \label{generalisation} Let $s \ge 9$, let $\tau$ be a real number, let $h_1, \ldots, h_s \in \bR[x]$ be cubic polynomials satisfying the irrationality condition, and put $H(\bx) = \sum_{i \le s} h_i(x_i)$. Then there exists $\bx \in \bZ^s$ such that
\[ |H(\bx) - \tau | < \eta. \]
\end{thm}

In order to assess the strength of Theorem \ref{Waring}, we consider what is known about Waring's problem. Linnik \cite{Lin1943} showed that any large positive integer can be expressed as a sum of at most seven positive cubes. Vaughan \cite{Vau1989jlms} later used smooth numbers to establish a lower bound, of the conjectured order of magnitude, for the number of representations. Both methods rely on arithmetic considerations which, due to the real shifts, are not useful in our problem. Consequently $s=8$ is a sensible target in Theorem 1, and our methods come agonisingly close to achieving this.

Cognoscenti will note that Vaughan also uses divisibility ideas to treat the eight variable case of Waring's problem in \cite{Vau1986Crelle}. However, if one does not seek an asymptotic formula, then the sixth moment estimate in Vaughan \cite{Vau1985}, which uses diminishing ranges, suffices to establish the existence of solutions. We will need to modify Vaughan's procedure for establishing low moment estimates, since divisibility cannot be used to study the underlying diophantine inequalities.

Theorem \ref{six} requires fewer variables than Waring's problem for cubes. Note, however, that any integer can be written as a sum of five integer cubes (see \cite[Theorem 405]{HW2008}). To prove Theorem \ref{six}, we use Linnik's idea \cite{Lin1943} to reduce to an indefinite, irrational ternary quadratic polynomial. We then invoke the work of Margulis and Mohammadi \cite{MM2011} on inhomogeneous quadratic polynomials.

Our overall strategy for proving Theorem \ref{Waring} is that of Freeman \cite{Fre2003}. We use the Davenport-Heilbronn method, with the treatment of the major arc being fairly standard. Next we perform a classical major and minor arc dissection using \cite[Theorem 5.1]{Bak1986}, which tells us that either a Weyl sum is small or its coefficients have good simultaneous rational approximations. Classical minor arcs are treated using a fourth moment estimate involving diminishing ranges. An $\eps$-free analogue to Hua's lemma is needed on classical major arcs, and a nontrivial bound is needed on Davenport-Heilbronn minor arcs. The former uses \cite[Lemma 4.4]{Bak1986}, while the latter is provided by \cite[Lemmas 8 and 9]{Fre2003}.

Theorem \ref{FullDensity} exemplifies the philosophy that if $2t$ variables suffice to solve an additive problem then $t$ variables suffice almost surely. We follow a recipe of Parsell and Wooley \cite{PW2013}. By considering the contributions from the Davenport-Heilbronn minor and trivial arcs in mean square, we may effectively work with $2t$ variables in this part of the analysis.

A remark made in the introduction of \cite{Fre2003} implies the conclusion of Theorem \ref{generalisation} whenever $s \ge 13$. Theorem \ref{generalisation} is obtained in the same way as Theorem \ref{Waring}. Our fourth moment estimate is slightly weaker in the general setting, but nonetheless permits a nine variable treatment. Following this same procedure, and then using the methods developed by Freeman \cite{Fre2002} and Wooley \cite{Woo2003}, yields Theorem \ref{AsymptoticFormula}. An asymptotic formula in fewer variables cannot be obtained via Lemma \ref{fourth1}, since the latter uses diminishing ranges. 

Margulis and Mohammadi \cite[Theorem 1.4]{MM2011} have shown that three variables suffice to give a version of Theorem \ref{generalisation} for quadratic polynomials. A simplification of our methods shows that $(k-1)2^{k-1}+3$ variables suffice for a degree $k$ diagonal analogue. In a similar vein we may obtain a degree $k \ge 2$ analogue to Theorem \ref{Waring} using $(k-1)2^{k-1}+3$ variables. In fact we can do much better; exponents $k \ge 4$ are vulnerable to a broader range of attacks, which we discuss in coming work. 

This paper is organised as follows. In \S \ref{prelim}, we establish the low moment estimates underpinning the proofs of our theorems, and also introduce work of Freeman which exploits the irrationality of $\mu_1$. In \S\S \ref{War}--\ref{sixvars} we prove Theorems \ref{Waring}, \ref{FullDensity}, \ref{LowerDensity}, \ref{AsymptoticFormula}, \ref{generalisation} and \ref{six} respectively.

We adopt the convention that $\eps$ denotes an arbitrarily small positive number, so its value may differ between instances. Bold face will be used for vectors, for instance we shall abbreviate $(x_1,\ldots,x_s)$ to $\bx$. We shall use the unnormalised sinc function, given by $\sinc(x) = \sin(x)/x$ for $x \in \bR \setminus \{0\}$ and $\sinc(0)=1$. We shall use $g(\alp)$ and $g_i(\alp)$ to denote Weyl sums, to be explicitly defined in each situation. 

The author thanks Trevor Wooley for suggesting this line of research, as well as for his dedicated supervision.

\section{Preliminary estimates}
\label{prelim}

Key inputs for this paper are sufficiently strong low moment estimates for Weyl sums. Let $P>0$ be a large real number. For real numbers $X>0, \alpha$ and $\mu$, write
\begin{equation} \label{fDefs}
f_j(\alpha, \mu, X) = \sum_{(j-1)X < x \le jX} e(\alpha(x-\mu)^3) \qquad (j=1,2).
\end{equation}
We note the identity
\begin{equation*}
f_2(\alp,\mu,X) = f_1(\alp,\mu,2X) - f_1(\alp,\mu,X),
\end{equation*}
from which we can deduce that $f_2$ inherits certain bounds from $f_1$. By considering the underlying diophantine equations, it is easy to see that
\[ \int_0^1 |f_j(\alpha,0,P)|^2 \d \alp \ll P \qquad (j=1,2) \]
and
\[
\int_0^1 |f_j(\alpha,0,P)|^4 \d \alp \ll P^{2+\eps} \qquad (j=1,2).
\]

We seek similar bounds for shifted cubes. First we introduce some notation, so as to delineate the relationship between moments of our Weyl sums and their associated diophantine inequalities. Put
\begin{equation} \label{Kdef}
K(\alpha) = K_\eta(\alp)= \eta \cdot \sinc^2(\pi \alp \eta).
\end{equation}
This kernel function was first used by Davenport and Heilbronn \cite{DH1946}. It satisfies
\begin{equation} \label{Kbound} 
0 \le K(\alp) \ll \min(1, |\alp|^{-2})
\end{equation} 
and, for any real number $t$,
\begin{equation}\label{orth1}
\int_\bR e(\alp t) K(\alp) \d \alp = \max(0, 1- |t/\eta|).
\end{equation}
Similarly
\begin{equation}\label{orth2}
4\int_\bR e(\alp t) K(2\alp) \d \alp = \max(0, 2- |t/\eta|).
\end{equation}
For $\kappa > 0$, we define the indicator function
\begin{equation} \label{ind}
U_\kappa(t)  = \begin{cases}
1, &\text{if } |t| < \kappa\\
0, &\text{if } |t| \ge \kappa.
\end{cases}
\end{equation}
By \eqref{orth1} and \eqref{orth2} we have
\begin{equation} \label{OrthBounds}
0 \le \int_\bR e(\alp t) K(\alp) \d \alp \le U_\eta(t) \le 4\int_\bR e(\alp t) K(2\alp) \d \alp \le 2 U_{2\eta}(t).
\end{equation}

\begin{lemma} \label{SecondMoment}
Let $h$ be a real polynomial of degree $d \ge 2$. Let $x$ and $y$ be integers such that $x,y > P$ and
\begin{equation*}
|h(x) - h(y)| < \eta.
\end{equation*}
Then $x=y$.
\end{lemma}

\begin{proof} The mean value theorem gives
\[ (h(x)-h(y)) / (x-y) \gg P^{d-1}, \]
so $|x-y| < 1$.
\end{proof}

This implies, for instance, that if $\mu \in \bR$ then
\[ \int_\bR |f_2(\alp, \mu, P)|^2 K(\alp) \d \alp \ll P. \]
We need to work harder for a fourth moment estimate. We shall use diminishing ranges. If $q > 0$ and $a$ are integers, let $\fR(q,a)$ be the set of $\alp \in \bR$ such that $|q \alp - a| \le P^{-3/2}$. Let $\fR$ be the union of the arcs $\fR(q,a)$ over $q \le P$ and $(a,q)=1$, and put $\fr = \bR \setminus \fR$. Let $\cR$ be the intersection of $\fR$ and a unit interval.

\begin{lemma} \label{fourth1}
Let $\mu_1$ and $\mu_2$ be real numbers. Then the number $S_4$ of integral solutions to
\begin{equation} \label{fourtheq}
|(x_1 - \mu_1)^3 - (y_1 - \mu_1)^3 + (x_2 - \mu_2)^3  - (y_2- \mu_2)^3 |  < \eta
\end{equation}
with $P< x_1,y_1 \le 2P$ and $P^{5/6} < x_2,y_2 \le 2P^{5/6}$
satisfies $S_4 \ll P^{11/6+\eps}$.
\end{lemma}

\begin{proof} We imitate Vaughan \cite{Vau1985}. By Lemma \ref{SecondMoment}, with $P^{5/6}$ in place of $P$, the number of solutions counted by $S_4$ with $x_1 = y_1$ is $O(P^{11/6})$. It therefore suffices to show that $S'_4 \ll P^{11/6+\eps}$, where $S'_4$ is the number of solutions counted by $S_4$ with $x_1 > y_1$. Write $y_1=x$ and $x_1 = x+ h$. The mean value theorem gives
\[
|(x_1 - \mu_1)^3 - (y_1 - \mu_1)^3| > (3 - \eps) P^2 |x_1 - y_1| = (3 - \eps) hP^2.
\]
By combining this with the inequalities \eqref{fourtheq} and
\[
|(x_2 - \mu_2)^3  - (y_2- \mu_2)^3| < (8 + \eps) P^{5/2},
\]
we deduce that $0 < h \le 3 P^{1/2}$. For integers $h$ and real numbers $\alp$, define
\[
\Phi_h(\alp) = \sum_{P < x \le 2P} e(\alp h(3x^2 + 3(h-2\mu_1)x + h^2 - 3 \mu_1 h + 3 \mu_1^2))
\]
and
\[
G(\alp) = \sum_{0 < h \le 3P^{1/2}} \Phi_h(\alp).
\]
From \eqref{OrthBounds} we have
\begin{equation} \label{fourthstart}
S'_4 \ll \int_\bR G(\alp) |f_2(\alp, \mu_2, P^{5/6})|^2 K(2\alp) \d\alp.
\end{equation}

Let $q > 0$ and $a$ be relatively prime integers such that $|q \alp -a| \le q^{-1}$. Following closely the proof of the lemma in \cite{Vau1985}, we now show that
\begin{equation} \label{Gbound0}
G(\alp) \ll P^\eps (P^{3/2}q^{-1/2} + P + P^{1/4}q^{1/2}).
\end{equation}
By Cauchy's inequality,
\[ |G(\alp)|^2 \ll P^{1/2} \sum_{0< h \le 3P^{1/2}} |\Phi_h(\alp)|^2. \]
Moreover,
\[ |\Phi_h(\alp)|^2 = \sum_{P < x \le 2P} \sum_{P < y \le 2P} e(3 \alp  h (x-y) ( x+y + h - 2 \mu_1)). \]
On writing $x = y + h_1$ this becomes
\[
|\Phi_h(\alp)|^2 = \sum_{|h_1| < P} \sum_y e(3\alp h h_1 (2y + h + h_1 - 2 \mu_1)),
\]
where the inner summation is over
\[ \max(P, P-h_1) < y \le \min(2P, 2P-h_1). \]
Now
\[
|\Phi_h(\alp)|^2 \ll P + \sum_{0 < h_1 < P} \min(P, \|6 \alp h h_1 \|^{-1}). 
\]
Therefore
\[
G(\alp)^2 \ll P^2 + P^{1/2+\eps} \sum_{0 < u < 18 P^{3/2}} \min(P, \| \alp u \|^{-1}).
\]
Applying \cite[Lemma 2.2]{Vau1997} now gives
\[ 
G(\alp)^2 \ll P^\eps (q^{-1} P^3 + P^2 + q P^{1/2})
\]
when $q \le P^3$, and \eqref{Gbound0} follows.

Let $\alp \in \fr$. By Dirichlet's approximation theorem \cite[Lemma 2.1]{Vau1997}, we may choose relatively prime integers $q>0$ and $a$ such that $q \le P^{3/2}$ and $|q \alp -a| \le P^{-3/2}$. Since $\alp \in \fr$, we must also have $q>P$. Now \eqref{Gbound0} gives
\[
G(\alp) \ll P^{1+\eps}.
\]
Moreover, applying Lemma \ref{SecondMoment} with $2\eta$ in place of $\eta$, and recalling \eqref{OrthBounds}, yields
\[ \int_\bR
|f_2(\alp, \mu_2, P^{5/6})|^2 K(2\alp) \d\alp \ll P^{5/6}.
\]
Thus,
\begin{align*}
 \int_\fr G(\alp) |f_2(\alp, \mu_2, P^{5/6})|^2 K(2\alp) \d\alp 
& \ll (\sup_{\alp \in \fr} |G(\alp)|) \cdot P^{5/6}
\\ & \ll P^{11/6+\eps}.
\end{align*}

In light of \eqref{Kbound} and \eqref{fourthstart}, it now remains to show that
\begin{equation} \label{MajorFourth}
\int_\cR |G(\alp) f_2(\alp, \mu_2, P^{5/6})^2|  \d\alp \ll P^{11/6+\eps}.
\end{equation}
Let $\alp \in \fR(q,a) \subseteq \fR$. Now \eqref{Gbound0} gives $G(\alp) \ll P^{3/2+\eps}q^{-1/2}$ and, by Weyl's inequality \cite[Lemma 2.4]{Vau1997}, we have
\[ 
f_2(\alp, \mu_2, P^{5/6}) \ll P^{5/6+\eps} (P^{-5/6}+q^{-1})^{1/4}. 
\]

The measure of $\fR(q,a)$ is $2q^{-1}P^{-3/2}$. Moreover, if $q \in \bN$ then there are at most $q+1$ integers $a$ such that $|q\alp - a| \le P^{-3/2}$ for some $\alp \in\cR$. Hence 
\begin{equation} \label{MajorFourth0}
 \int_\cR |G(\alp) f_2(\alp, \mu_2, P^{5/6})^2|  \d\alp \ll I_1+I_2, 
\end{equation}
where 
\begin{equation} \label{I1Fourth}
I_1 = P^{3/2+\eps} \sum_{q \le P} q^{-1/2} P^{2(5/6-5/24)-3/2} \ll P^{7/4+\eps} \ll P^{11/6}
\end{equation}
and
\begin{equation}\label{I2Fourth}
 I_2 = P^{3/2+\eps} \sum_{q \le P} q^{-1} P^{5/3-3/2} \ll P^{5/3+2\eps} \ll P^{11/6}.
\end{equation}
Substituting \eqref{I1Fourth} and \eqref{I2Fourth} into \eqref{MajorFourth0} implies \eqref{MajorFourth}, completing the proof.
\end{proof}

We give a slightly weaker bound for general cubic polynomials. This suffices for Theorem \ref{generalisation}, so for simplicity we do not give the strongest possible result.

\begin{lemma} \label{fourth2} Let $h_1, h_2 \in \bR[x]$ be cubic polynomials, and fix a real number $c > 1$. Then the number $S_4$ of integral solutions to
\begin{equation} \label{fourth2eq}
|h_1(x_1) - h_1(y_1) + h_2(x_2) - h_2(y_2) | < \eta
\end{equation}
with $P < x_1, y_1 \le cP$ and $P^{4/5} < x_2, y_2 \le 2P^{4/5}$ satisfies $S_4 \ll P^{9/5+\eps}$.
\end{lemma}

\begin{proof} We may assume without loss that $h_1$ is monic. By Lemma \ref{SecondMoment}, with $P^{4/5}$ in place of $P$, the number of solutions counted by $S_4$ with $x_1 = y_1$ is $O(P^{9/5})$. It therefore suffices to show that $S'_4 \ll P^{9/5+\eps}$, where $S'_4$ is the number of solutions counted by $S_4$ with $x_1 > y_1$. Write $y_1 = x$ and $x_1 = x+h$. Let $C$ be a large positive constant. The mean value theorem gives 
\[
|h_1(x_1) - h_1(y_1)| \gg P^2 |x_1 - y_1| = hP^2.
\]
By combining this with the inequalities \eqref{fourth2eq} and
\[
h_2(x_2)  - h_2(y_2) \ll P^{12/5},
\]
we deduce that $0 < h \le CP^{2/5}$. For integers $h$ and real numbers $\alp$, define
\[
\Phi_h(\alp) = \sum_{P < x \le cP} e(\alp h_1(x+h) - \alp h_1(x))
\]
and
\[
G(\alp) = \sum_{0 < h \le CP^{2/5}} \Phi_h(\alp).
\]
From \eqref{OrthBounds} we have
\begin{equation} \label{genFourthFinish}
S'_4 \ll \int_\bR G(\alp) |g(\alp)|^2 K(2\alp) \d\alp,
\end{equation}
where 
\[ g(\alp) = \sum_{P^{4/5} < x \le 2P^{4/5}} e(\alp h_2(x)). \]

Let $q > 0$ and $a$ be relatively prime integers such that $|q \alp - a| \le q^{-1}$. We now show that
\begin{equation} \label{Gbound}
G(\alp) \ll P^{1/5+\eps} \Bigl(        
\frac{P^{12/5}}{q+P^{12/5}|q \alp -a|} + P^{7/5} + q + P^{12/5}|q \alp - a|
\Bigr)^{1/2}.
\end{equation}
We initially follow the proof of the lemma in \cite{Vau1985}. We may plainly assume that $q \le P^3$. Put $M= P^{1/5}$, $H = P^{2/5}$ and $Q = P^{4/5}$. Let
\[ h_1(x) = x^3 + a_2x^2 + a_1x + a_0, \]
where $a_0, a_1, a_2 \in \bR$. Now
\[
\Phi_h(\alp) = \sum_{P < x \le cP} e(\alp h (3x^2 + 3hx + h^2  + a_2(2x +h) + a_1)).
\]

By Cauchy's inequality, we have
\begin{equation} \label{InitialCauchy}
|G(\alp)|^2 \ll H \sum_{0 < h \le CH} |\Phi_h(\alp)|^2.
\end{equation}
Moreover,
\[ |\Phi_h(\alp)|^2 = \sum_{P < x \le cP} \sum_{P < y \le cP} e(\alp  h (x-y) ( 3(x+y + h) +2a_2)). \]
On writing $x = y + d$ this becomes
\[
|\Phi_h(\alp)|^2 = \sum_{|d| < cP} \sum_y e(\alp h d (3(2y + h + d) +2 a_2)),
\]
where the inner summation is over
\[ \max(P, P- d) < y \le \min(cP, cP-d). \]
Thus,
\[
|\Phi_h(\alp)|^2 \ll P + \sum_{0 < d < cP} \min(P, \|6 \alp h d\|^{-1}),
\]
so
\[
\sum_{0< h \le CH} |\Phi_h(\alp)|^2 \ll HP + P^\eps \sum_{0 < u < 6CHcP} \min(P, \| \alp u \|^{-1}).
\]
Applying \cite[Lemma 2.2]{Vau1997} now gives
\begin{equation} \label{VaughanCompare}
\sum_{0 < h \le CH} |\Phi_h(\alp)|^2 \ll HP^{2+\eps}(q^{-1} + P^{-1} + q/(HP^2)).
\end{equation}

We may now apply the classical transference principle in \cite[\S 2.8, Exercise 2]{Vau1997} to deduce that
\begin{equation} \label{transference}
\sum_{0 < h \le CH}
|\Phi_h(\alp)|^2 \ll P^\eps \Bigl(        
\frac{P^2H}{q+Q^3|q \alp -a|} + PH + q + Q^3 |q \alp - a|
\Bigr).
\end{equation}
Here $q$ is replaced by $q + Q^3|q \alp - a|$, the latter being the `natural' height of $\alp$. For full details, see the proof of \cite[Lemma 3.1]{Vau1989acta}. The right hand side of \eqref{VaughanCompare} is precisely the right hand side of \cite[Equation (3.4)]{Vau1989acta}; there we put $r=q$, $b=a$ and $k=3$. 

Combining \eqref{transference} with \eqref{InitialCauchy} yields
\[ 
G(\alp) \ll H^{1/2}P^\eps \Bigl(        
\frac{P^2H}{q+Q^3|q \alp -a|} + PH + q + Q^3 |q \alp - a|
\Bigr)^{1/2},
\]
which establishes \eqref{Gbound}.

Let $\alp \in \fr$. By Dirichlet's approximation theorem, we may choose relatively prime integers $q>0$ and $a$ such that $q \le P^{3/2}$ and $|q \alp -a| \le P^{-3/2}$. Since $\alp \in \fr$, we must also have $q>P$. Now \eqref{Gbound} gives $G(\alp) \ll P^{19/20+\eps} \ll P$. Moreover, applying Lemma \ref{SecondMoment} with $2\eta$ in place of $\eta$, and recalling \eqref{OrthBounds}, yields
\begin{equation} \label{genSecond}
\int_\bR
|g(\alp)|^2 K(2\alp) \d\alp \ll P^{4/5}.
\end{equation}
Thus,
\begin{equation} \label{genFourthMinor}
 \int_\fr G(\alp) |g(\alp)|^2 K(2\alp) \d\alp \ll (\sup_{\alp \in \fr} |G(\alp)|) \cdot P^{4/5} \ll P^{9/5}.
\end{equation}

H\"older's inequality gives 
\begin{equation} \label{genFourthMajor0}
\int_\fR |G(\alp) g(\alp)^2| K(2\alp) \d\alp \le  J_1^{1/4} J_2^{3/4},
\end{equation}
where
\[ J_1 = \int_\fR |G(\alp)|^4 K(2\alp) \d \alp \]
and
\[ J_2 = \int_\fR |g(\alp)|^{8/3} K(2\alp)\d \alp. \]
From \eqref{genSecond} and a trivial estimate we have
\begin{equation} \label{genJ2}
J_2 \ll (P^{4/5})^{2/3}P^{4/5} = P^{4/3}.
\end{equation}

Note from \eqref{Gbound} that if $\alp \in \fR(q,a) \subseteq \fR$ then
\[ G(\alp) \ll q^{-1/2} P^{7/5+\eps} (1+P^{12/5}|\bet|)^{-1/2}, \]
where $\bet = \alp - a/q$. Moreover, if $q \in \bN$ then there are at most $q+1$ integers $a$ such that $|q\alp - a| \le P^{-3/2}$ for some $\alp \in\cR$. Hence 
\begin{align*}
\int_\cR |G(\alp)|^4 \d \alp 
&\ll \sum_{q \le P} q^{-1} \int_0^\infty P^{28/5+\eps} (1+P^{12/5}\bet)^{-2} \d \bet \\
&\ll P^{16/5+2\eps} 
\end{align*}
which, in light of \eqref{Kbound}, yields
\begin{equation} \label{genJ1}
J_1 \ll P^{16/5+\eps}.
\end{equation}
Substituting \eqref{genJ2} and \eqref{genJ1} into \eqref{genFourthMajor0} gives
\[
\int_\fR G(\alp) |g(\alp)|^2 K(2\alp) \d\alp \ll P^{9/5+\eps}.
\]
Combining this with \eqref{genFourthFinish} and \eqref{genFourthMinor} yields $S'_4 \ll P^{9/5+\eps}$, completing the proof.
\end{proof}

Lemmas \ref{fourth1} and \ref{fourth2} may be used for classical minor arcs, but we still require an epsilon-free analogue to Hua's lemma on classical major arcs. The key ingredient here is \cite[Theorem 5.1]{Bak1986}, which formalises the idea that if a Weyl sum is large then its coefficients have good simultaneous rational approximations. For $h \in \bR[x]$, and for real numbers $b \ge 0$ and $c > b$, we shall write
\[
g_{b,c}(\alp; h) = \sum_{bP < x \le cP} e(\alp h(x)).
\]
Note the identity
\[
g_{b,c}(\alp; h) = g_{0,c}(\alp; h) - g_{0,b}(\alp;h), 
\]
which will be used to infer certain bounds.

\begin{lemma} \label{seventh}
 Let $u > 6$ be a real number. Fix a cubic polynomial $h \in \bR[x]$, and fix $L > 0$. Let $0 \le b <c$, and let $g(\alp) = g_{b,c}(\alp; h)$. 
Put
\[ \fN = \{ \alp \in \bR: |g(\alp)| > P^{3/4+\eps} \},\]
and let $\fU$ be the intersection of $\fN$ with an interval of length $L$. Then
\begin{equation} \label{seventhEq}
\int_\fU |g(\alp)|^u \d \alp \ll_{h,L} P^{u-3}.
\end{equation}
\end{lemma}

\begin{proof} By changing variables, we may assume without loss that $h$ is monic. Let
\[ h(x) = x^3 + a_2x^2 + a_1x + a_0, \]
where $a_0, a_1, a_2 \in \bR$. For $\alp \in \bR$, put $\alp_j = a_j \alp$ ($j=0,1,2$) and $\alp_3 = \alp$. For $q \in \bN$ and $\bv \in \bZ^3$, put 
\[ 
S(q,\bv) = \sum_{x=1}^q e((v_3x^3+v_2x^2+v_1x)/q).
\]
For $\bet_0, \bet_1, \bet_2, \bet_3 \in \bR$, put
\[
I(\bbeta) = \int_{bP}^{cP} e(\bet_3 x^3 + \beta_2x^2 + \bet_1 x + \bet_0) \d x.
\]

Let $\alp \in \fU$. At least one of $|g_{0,b}(\alp; h)|$ and $|g_{0,c}(\alp; h)|$ must exceed $\frac12 P^{3/4+\eps}$. Thus, by \cite[Theorem 5.1]{Bak1986}, there exist integers $q,v_3,v_2$ and $v_1$ such that 
\[
0< q < P^{3/4},
\]
\begin{equation} \label{coprimality}
(q, v_3, v_2, v_1) = 1, \qquad (q,v_3,v_2) < P^\eps
\end{equation}
and
\begin{equation}  \label{approx}
| q \alp_j - v_j| < P^{3/4-j} \qquad (j=1,2,3).
\end{equation}
For positive integers $q$, let $V(q)$ denote the set of $\bv \in \bZ^3$ satisfying \eqref{coprimality}. For $\bv \in \bZ^3$ and $q \in \bN$, denote by $\fU(q,\bv)$ the set of $\alp \in \fU$ satisfying \eqref{approx}. 

Let $\alp \in \fU(q,\bv)$, where $q < P^{3/4}$ and $\bv \in V(q)$. Let
\[
\beta_j = \alp_j - q^{-1}v_j \qquad (j=1,2,3)
\]
and $\beta_0 = \alp_0$. It follows from \cite[Lemma 4.4]{Bak1986} that
\[
g(\alp) - q^{-1}S(q,\bv)I(\bbeta) \ll q^{2/3}P^\eps.
\]
Now \cite[Theorems 7.1 and 7.3]{Vau1997} give
\begin{align*}  g(\alp) &\ll q^{2/3}P^\eps + q^{\eps-1/3}P(1+P^3|\bet_3|)^{-1/3} 
\\ 
&\ll  q^{\eps-1/3}P(1+P^3|\bet_3|)^{-1/3}. 
\end{align*}

Specifying $q,v_3$ and $\bet_3$ determines $\alp$. Moreover, if $q \in \bN$ then there are $O_L(q)$ integers $v_3$ satisfying $|q \alp - v_3| < P^{-9/4}$ for some $\alp \in \fU$. Hence
\begin{align*}
\int_\fU |g(\alp)|^u \d \alp 
&\le \sum_{q < P^{3/4}} \sum_{\bv \in V(q)} \int_{\fU(q,\bv)}  |g(\alp)|^u \d \alp \\
&\ll \sum_{q < P^{3/4}} q^{\eps + 1 - u/3} \int_0^\infty P^u (1+P^3 \bet)^{-u/3} \d \bet.
\end{align*}
As $u > 6$ and $\eps$ is small, we now have \eqref{seventhEq}.
\end{proof}

We will require Freeman's bounds on Davenport-Heilbronn minor arcs. In \cite[Lemmas 8 and 9]{Fre2003}, the underlying variables lie in the range $(0,P]$. The same results hold, with the same proof, when the underlying variables lie in $(bP,cP]$ for some fixed real numbers $b \ge 0$ and $c > b$. We summarise thus.

\begin{lemma} \label{Freeman}
Let $\xi < 1$ be a positive real number, and let $h_1, h_2 \in \bR[y]$ be cubic polynomials satisfying the irrationality condition. Let $0 \le b < c$, and let
\[ 
g_i(\alp) = g_{b,c}(\alp; h_i), \qquad i=1,2. 
\]
Then there exists a positive real-valued function $T(P)$ such that 
\[ \lim_{P \to \infty} T(P) = \infty\]
and
\begin{equation} \label{FreemanEq}
\sup_{ P^{\xi-3} \le |\alp| \le T(P)} 
|g_1(\alp) g_2(\alp)|
\ll P^2T(P)^{-1}.
\end{equation}
\end{lemma}

This may appear stronger than Freeman's conclusion that
\begin{equation} \label{Freeman0} 
\sup_{ P^{\xi-3} \le |\alp| \le T(P)} |g_1(\alp) g_2(\alp)| = o(P^2).
\end{equation}
However, the bound \eqref{Freeman0} gives a positive real-valued function $T_1(P)$ such that 
\[ \lim_{P \to \infty} T_1(P) = \infty \]
and
\[
\sup_{ P^{\xi-3} \le |\alp| \le T(P)} |g_1(\alp) g_2(\alp)| \ll P^2 T_1(P)^{-1}.
\]
By putting $T_0(P) = \min(T(P),T_1(P))$, we obtain \eqref{FreemanEq} with $T_0(P)$ in place of $T(P)$. The advantage of \eqref{FreemanEq} over \eqref{Freeman0} is not seen until \S \ref{AF}.

\section{A Waring-type result}
\label{War}

In this section we prove Theorem \ref{Waring}. By fixing the variables $x_{10}, \ldots, x_s$ if necessary, we may plainly assume that $s=9$, and that
\begin{equation} \label{wlog} 0 \le \mu_1, \ldots, \mu_s < 1.\end{equation}
Let $\gam$ be a small positive real number. Define $P$ by $\tau = 7.1P^3$, and put 
\[ g_i(\alp) = 
\begin{cases}
f_2(\alp,\mu_i,P), & i=1,2,\ldots,7\\
f_2(\alp,\mu_i,P^{5/6}), & i= 8,9,
\end{cases}
\]
where we recall \eqref{fDefs}. By \eqref{OrthBounds}, we can show Theorem \ref{Waring} by establishing that
\[
\int_\bR g_1(\alp)\cdots g_9(\alp)e(-\alp\tau) K(\alp) \d \alp \gg P^{17/3}.
\]
Let $0 < \xi < 5/6$, and recall that $\mu_1 \notin \bQ$. With $T(P)$ as in Lemma \ref{Freeman}, applied to the polynomials $(x-\mu_1)^3$ and $(x-\mu_2)^3$, we define our Davenport-Heilbronn major arc by
\begin{equation} \label{dh1}
\fM = \{ \alp \in \bR : |\alp| \le P^{\xi - 3} \}, 
\end{equation}
our minor arcs by
\begin{equation} \label{dh2}
\fm = \{ \alp \in \bR : P^{\xi -3} < |\alp| \le T(P) \}, 
\end{equation}
and our trivial arcs by
\begin{equation} \label{dh3}
\ft = \{ \alp \in \bR :  |\alp| > T(P) \}.
\end{equation}

\begin{lemma} \label{major1} We have
\[
\int_\fM g_1(\alp)\cdots g_9(\alp)e(-\alp\tau) K(\alp) \d \alp \gg P^{17/3}.
\]
\end{lemma}

\begin{proof}
For real numbers $X>0$ and $\alp$, let
\[
I(\alp,X) = \int_X^{2X} e(\alp x^3) \d x,
\]
and write
\[ I_i(\alp) = 
\begin{cases}
I(\alp,P), & i=1,2,\ldots,7\\
I(\alp,P^{5/6}), & i= 8,9.
\end{cases}
\]
Define 
\begin{align*} 
\cI^{(1)} &= \int_\fM g_1(\alp)\cdots g_9(\alp)e(-\alp\tau) K(\alp) \d \alp, \\
\cI^{(2)} &= \int_\fM I_1(\alp)\cdots I_9(\alp)e(-\alp\tau) K(\alp) \d \alp 
\end{align*}
and
\[ \cI^{(3)} = \int_\bR I_1(\alp)\cdots I_9(\alp)e(-\alp\tau) K(\alp) \d \alp. \]

By \cite[Lemma 4.4]{Bak1986}, if $\alp \in \fM$, $X \in (0,P]$ and $1 \le i \le s$ then
\[
f_2(\alp, \mu_i, X) = \int_X^{2X} e(\alp(x-\mu_i)^3)\d x+ O(1) = I(\alp,X)+O(1).
\]
Recalling \eqref{Kbound}, we now conclude that
\begin{equation} \label{d1}
\cI^{(1)} - \cI^{(2)} \ll 
P^{\xi-3}P^{7+5/6} = o(P^{17/3}),
\end{equation}
since $\xi < 5/6$. By \cite[Theorem 7.3]{Vau1997} we have
\[ I(\alp,P) \ll |\alp|^{-1/3}, \]
and now \eqref{Kbound} and a trivial estimate yield
\begin{equation} \label{d2} \cI^{(3)} - \cI^{(2)} \ll  P^{5/3} \int_{P^{\xi-3}}^\infty \alp^{-7/3} \d \alp
= o(P^{17/3}).
\end{equation}

With $\cR_0 = (P,2P]^7 \times (P^{5/6},2P^{5/6}]^2$, consider
\[
\cI^{(3)} = \int_{\cR_0} \int_\bR e(\alp(x_1^3 + \ldots + x_s^3 -\tau)) K(\alp) \d \alp \d \bx.
\]
By \eqref{orth1}, changing variables gives
\begin{equation} \label{I3first}
\cI^{(3)} \gg \int_{\cR_1} (\eta - |y_1 + \ldots + y_9 - \tau|) \cdot (y_1 \cdots y_9)^{-2/3} \d \by,
\end{equation}
where $\cR_1$ is the set of $\by \in \bR^9$ such that
\begin{align}
\notag P^3 &< y_1, \ldots, y_7 \le 8P^3, \\
\label{y89} P^{5/2} &< y_8, y_9 \le 8P^{5/2} 
\end{align}
and
\[ 
|y_1 + \ldots + y_9 - \tau| < \eta. 
\]
Let $\cV$ denote the set of $\by \in \cR_1$ such that 
\begin{equation} \label{restrict}
 P^3 < y_2, y_3, \ldots, y_7 \le 1.01P^3
\end{equation}
and
\[
|y_1 + \ldots + y_9 - \tau| < \eta/2. 
\]
By positivity of the integrand in \eqref{I3first}, we have
\[
\cI^{(3)} \gg (P^{7 \times 3+2 \times 5/2})^{-2/3} \cdot \meas(\cV) = P^{-52/3} \cdot \meas(\cV).
\]

Since $\tau = 7.1 P^3$, we have
\[ 
P^3 +\eta < \tau - y_2 - \ldots - y_9 < 1.1P^3
\]
whenever the inequalities \eqref{y89} and \eqref{restrict} are satisfied. Hence
\[ 
\meas(\cV) \gg P^{6 \times 3 + 2 \times 5/2} = P^{23},
\]
so
\begin{equation} \label{I3}
\cI^{(3)} \gg P^{-52/3} P^{23}  = P^{17/3}.
\end{equation}
The bounds \eqref{d1}, \eqref{d2} and \eqref{I3} yield the desired result 
\[ 
\cI^{(1)} \gg P^{17/3}. 
\]
\end{proof}

By Lemma \ref{major1}, H\"older's inequality and symmetry, it remains to show that
\begin{equation} \label{rest1}
\int_{\fm \cup \ft} |g_1(\alp) g_2(\alp)g_3(\alp)^5 g_8(\alp)^2| K(\alp) \d \alp = o(P^{17/3}).
\end{equation}
Fix $i \in \{1,2,3\}$, let
\[ \fN = \{ \alp \in \bR: |g_i(\alp)| > P^{3/4+\eps} \}, \]
put $\fn = \bR \setminus \fN$, and let $\fU$ be the intersection of $\fN$ with a unit interval. In view of \eqref{OrthBounds}, Lemma \ref{fourth1} gives
\[
\int_\bR |g_i(\alp)^2 g_8(\alp)^2| K(\alp) \d \alp \ll P^{11/6+\eps}.
\]
Since $\gam$ and $\eps$ are small, we now have
\begin{align}
\notag
\int_\fn |g_i(\alp)^{7-\gam} g_8(\alp)^2| K(\alp) \d \alp  
&\ll (\sup_{\alp \in \fn} |g_i(\alp)| )^{5-\gam} P^{11/6+\eps} \\
\label{ClassicalMinor}
&\le P^{3(5-\gam)/4+11/6+2\eps}= o(P^{17/3-\gam}).
\end{align}

Lemma \ref{seventh} and a trivial estimate yield
\[
\int_\fU |g_i(\alp)^{7-\gam} g_8(\alp)^2 | \d \alp \ll P^{17/3-\gam}
\]
which, recalling \eqref{Kbound}, gives
\begin{equation} \label{MinorMajor}
\int_{\fN}  |g_i(\alp)^{7-\gam} g_8(\alp)^2 | K(\alp) \d \alp \ll P^{17/3-\gam}
\end{equation}
and
\begin{align} \notag
\int_{\ft \cap \fN} |g_i(\alp)^{7 - \gam} g_8(\alp)^2| K(\alp) \d \alp
&\ll P^{17/3 - \gam} \sum_{n=0}^\infty (T(P)+n)^{-2} 
\\ \label{TrivialMajor} &= o(P^{17/3 - \gam} ).
\end{align}
The inequalities \eqref{ClassicalMinor} and \eqref{TrivialMajor}, together with a trivial estimate and H\"older's inequality, yield
\begin{equation} \label{trivial}
\int_\ft |g_1(\alp) g_2(\alp) g_3(\alp)^5 g_8(\alp)^2 | K(\alp) \d \alp = o(P^{17/3}).
\end{equation}

Combining \eqref{ClassicalMinor} with \eqref{MinorMajor} gives
\[ \int_\bR |g_i(\alp)^{7-\gam} g_8(\alp)^2| K(\alp) \d \alp \ll P^{17/3 - \gam} \]
for $i=1,2,3$ which, by H\"older's inequality and \eqref{Freeman0}, yields
\[ \int_\fm |g_1(\alp) g_2(\alp)g_3(\alp)^5 g_8(\alp)^2| K(\alp) \d \alp = o(P^{17/3}).\]
This and \eqref{trivial} give \eqref{rest1}, completing the proof of Theorem \ref{Waring}.

\section{A full density result}
\label{FD}

In this section we prove Theorem \ref{FullDensity}. Note that $Z(N)$ is closed and hence measurable, since it is the intersection of $[0,N]$ with
\[ \cap_\bx ((-\infty, F(\bx)-\eta] \cup [F(\bx)+\eta,\infty)),\]
where the intersection is taken over integers $x_1 > \mu_1, \ldots, x_s > \mu_s$. We may plainly assume \eqref{wlog} and, by fixing the variables $x_6, \ldots, x_s$ if necessary, that $s=5$. Put $\lambda = 42/41$. It suffices to show that
\[ \meas(Z(\lambda^{-1}X,X)) = o(X); \]
indeed, if for some $\psi(X) \nearrow \infty$ we have
\[ \meas(Z(\lambda^{-1}X,X)) \le \psi(X)^{-1}X \]
for large positive real numbers $X$, then
\begin{align*}
\meas(Z(N)) &\le \sqrt N+\sum_{i=0}^{ \lceil 0.5 \log_\lambda N \rceil - 1} \frac N{\lambda^i \psi(N/\lambda^i)} \\
&\ll \sqrt{N} + \psi(\sqrt{N})^{-1}N = o(N).
\end{align*}

Let $Z = Z(4.1N, 4.2N)$, and put $\hZ = \meas(Z)$. It remains to show that
\begin{equation} \label{GoalFour}
\hZ = o(N).
\end{equation}
Let $P = N^{1/3}$, and put
\[ 
\tf(\alp) = f_2(\alp, \mu_5, P^{5/6}) \prod_{i \le 4} f_2(\alp,\mu_i,P).
\]
By \eqref{OrthBounds} and \eqref{wlog}, we note that if $\tau \in Z$ then
\begin{equation} \label{zero}
\int_\bR \tf(\alp) e(-\alp \tau) K(\alp) \d \alp =0.
\end{equation}

Let $0 < \xi < 5/6$, and recall that $\mu_1 \notin \bQ$. With $T(P)$ as in Lemma \ref{Freeman}, applied to the polynomials $(x-\mu_1)^3$ and $(x-\mu_1)^3$, we define our Davenport-Heilbronn arcs by \eqref{dh1}, \eqref{dh2} and \eqref{dh3}. The inequality \eqref{rest1}, together with symmetry, H\"older's inequality and a trivial estimate, gives
\begin{equation} \label{MSbound}
\int_{\fm \cup \ft} |\tf(\alp)|^2 K(\alp) \d \alp = o(P^{20/3}).
\end{equation}
By mimicking the proof of Lemma \ref{major1}, one may confirm that
\begin{equation} \label{Major4}
\int_\fM \tf(\alp) e(-\alp \tau) K(\alp) \d \alp \gg P^{11/6}
\end{equation}
uniformly for $\tau \in [4.1N, 4.2N]$. Indeed, an inspection of that argument shows that the only detail requiring attention is the analogue of \eqref{d2}, which in the current setting becomes
\[ \cI^{(3)} - \cI^{(2)} \ll  P^{5/6} \int_{P^{\xi-3}}^\infty \alp^{-4/3} \d \alp
= o(P^{11/6}).
\]

Put
\[ T= \int_Z \int_{\fm \cup \ft} \tf(\alp) e(-\alp\tau) K(\alp)\d \alp \d \tau. \]
From \eqref{zero} and \eqref{Major4}, we have
\[
\int_{\fm \cup \ft} \tf(\alp) e(-\alp \tau) K(\alp) \d \alp \gg P^{11/6}
\]
uniformly for $\tau \in Z$, so
\begin{equation} \label{A}
T \gg \hZ P^{11/6}.
\end{equation}
For $\tau \in Z$, define the complex number $\theta_\tau$ by
\begin{equation*}
\Bigl |\int_{\fm \cup \ft} \tf(\alp) e(-\alp \tau) K(\alp) \d \alp \Bigr | = 
\theta_\tau \int_{\fm \cup \ft} \tf(\alp) e(-\alp \tau) K(\alp) \d \alp, 
\end{equation*}
and note that $|\theta_\tau|=1$. For $\alp \in \bR$ let
\[ H(\alp) = \int_Z \theta_\tau e(-\alp\tau) \d \tau,\]
so that
\begin{align} \notag
|T| & \le  \int_Z  \Biggl | \int_{\fm \cup \ft}\tf(\alp) e(-\alp \tau) K(\alp) \d \alp  \Biggr  | \d \tau
\\ \label {FubiniIneq} &=\int_{\fm \cup \ft} \tf(\alp) H(\alp) K(\alp) \d \alp .
\end{align}

Let
\[ J= \int_\bR |H(\alp)|^2 K(\alp) \d \alp.\]
Applying Fubini's theorem and \eqref{OrthBounds} yields
\begin{equation} \label{Jbound}
J = \int_{Z \times Z} \theta_\tau \overline{\theta_\nu} \int_\bR e(\alp(\nu-\tau))K(\alp) \d \alp \d \tau \d \nu \ll \hZ.
\end{equation}
By Cauchy's inequality,
\[ \Bigl |\int_{\fm \cup \ft} \tf(\alp)H(\alp)K(\alp) \d \alp \Bigr |^2 \ll J \int_{\fm \cup \ft} |\tf(\alp)|^2 K(\alp) \d \alp. \]
Substituting the bound \eqref{FubiniIneq} into the left hand side, and substituting the bounds \eqref{MSbound} and \eqref{Jbound} into the right hand side, yields
\begin{equation} \label{B}
T^2
=o(\hZ P^{20/3}).
\end{equation}
By \eqref{A} and \eqref{B}, we have $\hZ = o(P^3)$. This gives \eqref{GoalFour}, completing the proof of Theorem \ref{FullDensity}.

\section{Unrepresentation}
\label{LD}

In this section we prove Theorem \ref{LowerDensity}. For $\bx \in \bZ^3$ and $\tau \in \bR$, let
\[ \iota(\bx, \tau) = 
\begin{cases}
1, & \text{if } |F(\bx) - \tau| < \eta \\
0, & \text{otherwise}
\end{cases}
\]
and
\[ I_\bx = \int_0^N \iota (\bx, \tau) \d \tau. \]
Let $N(\tau)$ be as defined in the preamble to Theorem \ref{AsymptoticFormula}, and put $\hZ = \meas(Z(N))$. Then
\[ N - \hZ \le \int_0^N N(\tau) \d \tau = \sum_\bx I_\bx, \]
where the summation is over integer triples $x_1 > \mu_1$, $x_2 > \mu_2$, $x_3 > \mu_3$. 

Note that $I_\bx \le 2 \eta$. Moreover, if $\bx \in (\mu_1, \infty) \times (\mu_2, \infty) \times (\mu_3, \infty)$ and $\iota(\bx, \tau) = 1$ for some $\tau \in [0,N]$ then
\[ F(\bx + \br) < N + 10N^{2/3} \]
for all $\br \in [0,1]^3$. Hence
\begin{align*} N - \hZ &\le 2 \eta \cdot \meas (\{ \bgam \in \bR_{>0}^3 :  \gam_1^3 + \gam_2^3 + \gam_3^3 < N + 10N^{2/3} \} )\\
& = 2 \eta (N+10N^{2/3}) \cdot \meas (\{ \bgam \in \bR_{>0}^3 :  \gam_1^3 + \gam_2^3 + \gam_3^3 < 1 \}).
\end{align*}
Since $N$ is large and $\eta < 1/4$, we now have $N - \hZ < N/2$, so $\meas(Z(N)) > N/2$.

\section{An asymptotic formula}
\label{AF}

In this section we prove Theorem \ref{AsymptoticFormula}. Let $\tau$ be a large positive real number, and put 
\begin{equation} \label{Pdef}
P=\tau^{1/3}.
\end{equation}
We may plainly assume \eqref{wlog}. One can easily check that
\[ N(\tau) - N^*(\tau) \ll P^{2(s-1)/3} = o(\tau^{s/3-1}), \]
where $N^*(\tau)$ is the number of integral $\bx \in [1,P]^s$ satisfying \eqref{main}. It therefore suffices to prove the theorem with $N^*(\tau)$ in place of $N(\tau)$.

For $i=1,2,\ldots,s,$ write
\[ g_i(\alp) = f_1(\alp, \mu_i, P), \]
where we cast our minds back to \eqref{fDefs}. With $0 < \xi < 1$, and with $T(P)$ as in Lemma \ref{Freeman}, applied to the polynomials $(x-\mu_1)^3$ and $(x- \mu_2)^3$, we define our Davenport-Heilbronn arcs by \eqref{dh1}, \eqref{dh2} and \eqref{dh3}.

Next we deploy a kernel function introduced in \cite[\S 2.1]{Fre2002}. Put 
\begin{equation} \label{Ldef}
L(P) = \min(\log T(P), \log P), \qquad \delta = \eta L(P)^{-1}
\end{equation}
and
\[ K_{\pm}(\alp) = \frac {\sin(\pi \alp \delta) \sin(\pi \alp(2 \eta \pm \delta))} {\pi^2 \alp^2 \delta}. \]
From \cite[Lemma 1]{Fre2002} and its proof, we have
\begin{equation} \label{Kbounds}
K_\pm(\alp) \ll \min(1, L(P) |\alp|^{-2})
\end{equation}
and
\begin{equation} \label{Ubounds}
0 \le \int_\bR e(\alp t) K_{-}(\alp)\d\alp \le U_\eta(t) \le \int_\bR e(\alp t) K_{+}(\alp)\d\alp \le 1,
\end{equation}
where we recall the definition \eqref{ind}. Moreover, the expression
\[ \Bigl|  \int_\bR e(\alp t) K_\pm(\alp) \d \alp - U_\eta(t) \Bigr| \]
is less than or equal to 1, and is equal to 0 whenever $| | t|- \eta| > \eta L(P)^{-1}$. 

It will be convenient to work with nonnegative kernels in part of the analysis, as in \cite[\S 2]{PW2013}. We note that
\begin{equation} \label{decompose}
|K_{\pm}(\alp)|^2 = K_1(\alp)K_2^{\pm}(\alp),
\end{equation}
where
\[ K_1(\alp) = \sinc^2(\pi \alp \delta) \]
and
\[ K_2^{\pm}(\alp) = (2\eta \pm \delta)^2 \sinc^2(\pi \alp (2\eta \pm \delta)). \]
As \eqref{OrthBounds} holds for all $\eta > 0$, we also have
\begin{equation} \label{aux1}
0 \le \int_\bR e(\alp t) K_1(\alp) \d \alp \le \delta^{-1} U_\delta(t) \ll L(P) \cdot U_\delta(t)
\end{equation}
and
\begin{equation} \label{aux2}
0 \le \int_\bR e(\alp t) K_2^{\pm}(\alp) \d \alp \le (2\eta \pm \delta) U_{2\eta \pm \delta}(t) \ll U_{2 \eta \pm \delta}(t).
\end{equation}

From \eqref{Ubounds} we have
\[
R_{-}(P)  \le N^*(\tau) \le R_+(P), 
\]
where 
\[ R_\pm(P) = \int_\bR g_1(\alp) \cdots g_s(\alp) e(-\alp \tau) K_\pm(\alp) \d \alp. \]
It therefore remains to show that
\begin{equation} \label{goal3} 
R_\pm(P) = 2 \eta \Gamma(4/3)^s \Gamma(s/3)^{-1} P^{s-3} + o(P^{s-3}).
\end{equation}
We begin by demonstrating the bound
\begin{equation} \label{rest3}
\int_{\fm \cup \ft} g_1(\alp) \cdots g_s(\alp) e(-\alp \tau)K_\pm(\alp) \d \alp = o(P^{s-3}).
\end{equation}
For this purpose it suffices, by symmetry, H\"older's inequality and a trivial estimate, to prove that
\begin{equation} \label{rest3alt}
\int_{\fm \cup \ft} |g_1(\alp) g_2(\alp) g_3(\alp)^9 K_\pm(\alp)| \d \alp = o(P^8).
\end{equation}

Let
\[ \fN = \{ \alp \in \bR : |g_3(\alp)| > P^{3/4+\eps} \}, \]
put $\fn = \bR \setminus \fN$, and let $\fU$ be the intersection of $\fN$ with a unit interval. For subsets $U \subseteq \bR$, write 
\[ \cI_\pm(U) = \int_U |g_1(\alp) g_2(\alp) g_3(\alp)^9 K_\pm(\alp)| \d \alp.\]
By \eqref{aux1}, \eqref{aux2} and Lemma \ref{SecondMoment}, we have
\[ \int_\bR |g_1(\alp)|^2 K_1(\alp) \d \alp \ll P \cdot L(P) \]
and
\[ \int_\bR |g_2(\alp)|^2 K_2^{\pm}(\alp) \d \alp \ll P. \]
Cauchy's inequality and \eqref{decompose} now give
\[ \int_\bR |g_1(\alp)g_2(\alp)K_{\pm}(\alp)| \d \alp \ll P \cdot L(P)^{1/2}. \]
Thus, recalling \eqref{Ldef}, we have
\begin{align} \notag
\cI_\pm(\fn) &\ll  (\sup_{\alp \in \fn} |g_3(\alp)|)^9 P \cdot L(P)^{1/2} \\
 \label{ClassicalMinor3} &\le P^{9 \times 3/4 + \eps + 1} L(P)^{1/2} = o(P^8).
\end{align}

By Lemma \ref{seventh}, we have
\begin{equation} \label{ninth}
\int_\fU |g_3(\alp)|^9 \d \alp \ll P^6.
\end{equation}
Combining this with \eqref{FreemanEq} gives
\begin{align*}
\int_{\fm \cap \fU} |g_1(\alp)g_2(\alp)g_3(\alp)^9| \d \alp 
&\ll (\sup_{\alp \in \fm} |g_1(\alp)g_2(\alp)|) \cdot P^6 \\
& \ll P^8 T(P)^{-1}
\end{align*}
which, recalling \eqref{Ldef} and \eqref{Kbounds}, yields
\begin{equation} \label{MinorMajor3}
\cI_\pm(\fm \cap \fN) \ll P^8 T(P)^{-1} L(P) = o(P^8).
\end{equation}
Moreover, \eqref{ninth} and a trivial estimate give
\[
\int_\fU |g_1(\alp)g_2(\alp)g_3(\alp)^9| \d \alp \ll P^8,
\]
so by \eqref{Ldef} and \eqref{Kbounds} we have
\begin{equation} \label{TrivialMajor3}
\cI_\pm(\ft \cap \fN) \ll P^8\sum_{n=0}^\infty L(P)\cdot (T(P)+n)^{-2}  = o(P^8).
\end{equation}
Since 
\[
\fm \cup \ft \subseteq \fn \cup (\fm \cap \fN) \cup (\ft \cap \fN),
\]
the inequalities \eqref{ClassicalMinor3}, \eqref{MinorMajor3} and \eqref{TrivialMajor3} give \eqref{rest3alt}, which in particular establishes \eqref{rest3}.

Next we consider
\begin{equation} \label{I1def}
\cI_\pm^{(1)} = \int_\fM g_1(\alp) \cdots g_s(\alp) e(-\alp \tau) K_\pm (\alp) \d \alp, 
\end{equation}
following the recipe given in \cite[\S 6]{Woo2003}. Define
\[ I(\alp) = \int_0^P e(\alp x^3) \d x,\]
\[ \cI_\pm^{(2)} = \int_\fM I(\alp)^s e(-\alp \tau) K_\pm (\alp) \d \alp \]
and 
\[ \cI_\pm^{(3)} = \int_\bR I(\alp)^s e(-\alp \tau) K_\pm (\alp) \d \alp. \]
Using \eqref{Kbounds} in place of \eqref{Kbound}, we may mimic the proofs of \eqref{d1} and \eqref{d2} to deduce that 
\begin{equation} \label{d123}
\cI_\pm^{(1)} -  \cI_\pm^{(2)} = o(P^{s-3})
\end{equation}
and
\begin{equation} \label{d233}
\cI_\pm^{(2)} - \cI_\pm^{(3)} = o(P^{s-3}).
\end{equation}

The final step is to provide asymptotics for
\[ \cI_\pm^{(3)} = \int_{(0,P]^s} \int_\bR  e(\alp(x_1^3+\ldots+x_s^3- \tau)) K_\pm (\alp) \d \alp \d \bx. \]
Changing variables with $u_i = P^{-3} x_i^3$ ($1 \le i \le s$) yields
\[ \cI_\pm^{(3)} = 3^{-s}P^s \int_{(0,1]^s} (u_1 \cdots u_s)^{-2/3} \Delta_\pm(\bu) \d \bu, \]
where 
\[ \Delta_\pm(\bu) =
\int_\bR e(\alp(P^3(u_1+\ldots+u_s)-\tau))K_\pm(\alp) \d \alp.
\]
Put
\[
\Delta^*(\bu) = \begin{cases}
1, &\text{if } |u_1+\ldots+u_s -1| < \eta P^{-3} \\
0,& \text{if } |u_1+\ldots+u_s -1| \ge \eta P^{-3}
\end{cases}
\]
and
\[ I^* = \int_{(0,1]^s} (u_1 \cdots u_s)^{-2/3} \Delta^*(\bu) \d \bu. \]

In view of \eqref{Pdef} and the discussion following \eqref{Ubounds}, we see that
\[ \Delta_\pm(\bu) = \Delta^*(\bu), \]
except possibly when
\begin{equation} \label{except}
||u_1+\ldots+u_s-1| - \eta P^{-3}| \le \eta P^{-3} L(P)^{-1},
\end{equation}
in which case $|\Delta_\pm(\bu) - \Delta^*(\bu)| \le 1$. If \eqref{except} is satisfied then there exists $j \in \{1,2,\ldots, s\}$ such that $u_j \gg 1$. For $j=1,2,\ldots,s$, let $T_j$ denote the set of $\bu \in (0,1]^s$ satisfying \eqref{except} and $u_j \gg 1$. Now
\[
\int_{T_j} (u_1 \cdots u_s)^{-2/3} \d \bu \ll P^{-3} L(P)^{-1} \qquad (1 \le j \le s),
\]
so
\[ \int_{(0,1]^s} (u_1 \cdots u_s)^{-2/3} (\Delta_\pm(\bu) - \Delta^*(\bu)) \d \bu = o(P^{-3}). \]
Thus,
\begin{equation} \label{IndDif}
\cI_\pm^{(3)}-3^{-s}P^s  I^* = o(P^{s-3}).
\end{equation}

For $\bu \in (0,1]^s$, write $\bu' = (u_1,\ldots,u_{s-1})$ and $Y = 1- u_1 - \ldots - u_{s-1}$. For $S \subseteq (0,1]^s$, define
\[ I(S) = \int_S (u_1 \cdots u_s)^{-2/3} \Delta^*(\bu) \d \bu. \]
Let $I_1 = I((0,1]^{s-1} \times (0,P^{-1}))$ and $I_2 = I((0,1]^{s-1} \times [P^{-1},1])$,
so that 
\begin{equation} \label{tweak}
I^* = I_1 + I_2.
\end{equation}
First we show that 
\begin{equation} \label{3I1}
I_1 = o(P^{-3}).
\end{equation}
Since $\int_0^{P^{-1}} u_s^{-2/3} \d u_s = o(1)$, it suffices for \eqref{3I1} to show that
\begin{equation} \label{3I1g}
\int_{(0,1]^{s-1}} (u_1 \cdots u_{s-1})^{-2/3} \Delta^*(\bu) \d \bu' \ll P^{-3},
\end{equation}
uniformly for $u_s \in (0, P^{-1})$. Let $0 < u_s < P^{-1}$. If $\Delta^*(\bu) = 1$ then there exists $j \in \{1,2,\ldots,s-1\}$ such that $u_j \gg 1$. For $j = 1,2,\ldots,s-1$, let $R_j$ denote the set of $\bu' \in (0,1]^{s-1}$ such that $\Delta^*(\bu)=1$ and $u_j \gg 1$. Now
\[ \int_{R_j} (u_1 \cdots u_{s-1})^{-2/3} \d \bu' \ll P^{-3} \qquad (1 \le j \le s-1),\]
which establishes \eqref{3I1g} and in particular \eqref{3I1}.

If $\Delta^*(\bu) = 1$ and $u_s \ge P^{-1}$ then $|u_s - Y| < \eta P^{-3}$ so, by the mean value theorem,
\[ u_s^{-2/3} - Y^{-2/3} \ll (P^{-1})^{-5/3} P^{-3}
= P^{-4/3}. \]
Combining this with the bound
\[ \int_{(0,1]^{s-1}} (u_1 \cdots u_{s-1})^{-2/3}  \int_{P^{-1}}^1 \Delta^*(\bu) \d u_s \d \bu' \ll P^{-3} \]
gives
\begin{equation} \label{3I23}
I_2 - I_3 = o(P^{-3}),
\end{equation}
where 
\[ I_3 = \int_{(0,1]^{s-1}} (u_1 \cdots u_{s-1}Y)^{-2/3}  \int_{P^{-1}}^1 \Delta^*(\bu) \d u_s \d \bu'. \]
Let $R$ be the set of $\bu' \in (0,1]^{s-1}$ such that $Y > 0$. As $|u_s - Y| < \eta P^{-3}$ whenever $\Delta^*(\bu) \ne 0$, we have
\[
I_3 = \int_R (u_1 \cdots u_{s-1}Y)^{-2/3}  \int_{P^{-1}}^1 \Delta^*(\bu) \d u_s \d \bu'.
\]

Next we show that
\begin{equation} \label{3I34}
I_3 - I_4 = o(P^{-3}),
\end{equation}
where
\begin{align*} I_4 &= \int_R (u_1 \cdots u_{s-1}Y)^{-2/3}  \int_\bR \Delta^*(\bu) \d u_s \d \bu'  \\
&= 2 \eta P^{-3} \int_R (u_1 \cdots u_{s-1}Y)^{-2/3} \d \bu'.
\end{align*}
Let $\bu \in R \times \bR$ be such that $\Delta^*(\bu) = 1$. Then $|u_s - Y| < \eta P^{-3}$, so $u_s > -\eta P^{-3}$. If $u_s < P^{-1}$ then $Y < 2P^{-1}$ and $u_j \gg 1$ for some $j \in \{1,2,\ldots, s-1\}$, so we can change variables from $u_j$ to $Y$ to show that the contribution from these $\bu$ is $o(P^{-3})$. Meanwhile, if $u_s>1$ then $Y > 1 - \eta P^{-3}$ and $u_1, \ldots, u_{s-1} \ll P^{-3}$, so the contribution from these $\bu$ is also $o(P^{-3})$. We have established \eqref{3I34}.

The computation
\begin{align*}
\int_R (u_1 \cdots u_{s-1}Y)^{-2/3} \d \bu' 
&=  \underset{ u_1 + \ldots + u_{s-1} < 1 }{\int_0^1 \cdots \int_0^1} (u_1 \cdots u_{s-1}Y)^{-2/3} \d \bu'
\\  &= \Gamma(1/3)^s \Gamma(s/3)^{-1}
\end{align*}
is standard (see \cite[p. 22]{Dav2005}). Therefore
\[
I_4 = 2 \eta P^{-3} \Gamma(1/3)^s \Gamma(s/3)^{-1}.
\]
In view of \eqref{tweak}, \eqref{3I1}, \eqref{3I23} and \eqref{3I34}, we now have
\[ I^* = 2 \eta \Gamma(1/3)^s \Gamma(s/3)^{-1} P^{-3} + o(P^{-3}). \]
Combining this with \eqref{d123}, \eqref{d233} and \eqref{IndDif} yields
\begin{equation} \label{finally}
\cI_\pm^{(1)} = 2 \eta \Gamma(4/3)^s \Gamma(s/3)^{-1} P^{s-3} + o(P^{s-3}),
\end{equation}
where we recall \eqref{I1def}. Finally, \eqref{rest3} and \eqref{finally} give \eqref{goal3}, completing the proof of Theorem \ref{AsymptoticFormula}.

\section{Sums of cubic polynomials}
\label{gen}

In this section we prove Theorem \ref{generalisation}. Without loss of generality $h_1$ and $h_2$ satisfy the irrationality condition. By fixing the variables $x_{10}, \ldots, x_s$ if necessary, we may plainly assume that $s = 9$, that $\tau=0$, and that $\eta = 1$. Let $a_1, \ldots, a_9$ be the leading coefficients of $h_1, \ldots, h_9$ respectively. Without loss of generality 
\[ a_i < 0 < a_1 \le |a_2| \qquad (2 \le i \le 9). \]
Let $\omega$ be a small positive real number, and let $P$ be a large positive real number. Define the real number $c$ by
\[ c^3 = -8a_1^{-1}(a_2 + \ldots + a_7), \]
and note that $c \ge 2$. Put
 \[ g_i(\alp) = \begin{dcases}
\sum_{P < x \le cP} e(\alp h_i(x)), &i=1,2,\ldots,7 \\
\sum_{P^{4/5} < x \le 2P^{4/5}} e(\alp h_i(x)), &i=8,9. 
\end{dcases}
\]
Recall \eqref{Kdef}. In light of \eqref{OrthBounds}, it suffices to show that
\[
\int_\bR g_1(\alp) \cdots g_9(\alp) K(\alp) \d \alp \gg P^{28/5}.
\]
We essentially follow \S \ref{War}. With $0 < \xi < 4/5$, and with $T(P)$ as in Lemma \ref{Freeman}, we define our Davenport-Heilbronn arcs by \eqref{dh1}, \eqref{dh2} and \eqref{dh3}. 

\begin{lemma} \label{major5} We have
\[
\int_\fM g_1(\alp)\cdots g_9(\alp) K(\alp) \d \alp \gg P^{28/5}.
\]
\end{lemma}

\begin{proof}
Put
\[ I_i(\alp) = \begin{dcases}
\int_P^{cP} e(\alp h_i(x)) \d x, \qquad i=1,2,\ldots,7 \\
\int_{P^{4/5}}^{2P^{4/5}} e(\alp h_i(x)) \d x, \qquad i=8,9.
\end{dcases}
\]
Define 
\begin{align*} 
\cI^{(1)} &= \int_\fM g_1(\alp)\cdots g_9(\alp) K(\alp) \d \alp, \\
\cI^{(2)} &= \int_\fM I_1(\alp)\cdots I_9(\alp) K(\alp) \d \alp 
\end{align*}
and
\[ \cI^{(3)} = \int_\bR I_1(\alp)\cdots I_9(\alp) K(\alp) \d \alp. \]
Mimicking the proof of Lemma \ref{major1}, we deduce that 
\begin{equation} \label{d125}
\cI^{(1)} -  \cI^{(2)} = o(P^{28/5})
\end{equation}
and
\begin{equation} \label{d235}
\cI^{(2)} - \cI^{(3)} = o(P^{28/5}).
\end{equation}

Write $R = [P,cP]^7 \times [P^{4/5}, 2P^{4/5}]^2$, and consider
\[
\cI^{(3)} = \int_R \int_\bR e(\alp H(\bx)) K(\alp) \d \alp \d \bx.
\]
Let
\[ X = [(1+2\omega)P, (1+3\omega)P]^6 \times [P^{4/5}, 2P^{4/5}]^2. \]
By \eqref{orth1}, we have
\begin{equation} \label{spec}
\cI^{(3)} = \int_R \max(0, 1- |H(\bx)|) \d \bx \gg V,
\end{equation}
where $V$ is the measure of the set of $\bx \in [P,cP] \times X$ such that $|H(\bx)| \le 1/2$. In view of \eqref{d125}, \eqref{d235} and \eqref{spec}, it remains to show that
\begin{equation} \label{unif}
 \meas \{x_1 \in [P,cP] : |H(\bx)| \le 1/2 \} \gg P^{-2}, 
\end{equation}
uniformly for $(x_2, \ldots, x_9) \in X$.

Let $\bx' = (x_2, \ldots, x_9) \in X$, and put
\[ \Lambda (\bx') =  - 1/2 - \sum_{i =2}^9 h_i(x_i). \]
Then 
\begin{equation} \label{Lambound}
-(a_2 + \ldots + a_7)(1+\omega)^3 < P^{-3}\Lambda (\bx') <-(a_2+ \ldots + a_7)(1+4\omega)^3. 
\end{equation}
The polynomial $h_1$ is strictly increasing when its argument is sufficiently large. As $\Lambda(\bx')$ is large and positive, there exist unique positive real numbers $m$ and $M$ such that $h_1(m) = \Lambda(\bx')$ and $h_1(M) = \Lambda(\bx') + 1$. Since $a_1 \le |a_2|$ and $\omega$ is small, it follows from \eqref{Lambound} that
\[ 
P < m < M < cP.
\]
Any $x_1 \in [m,M]$ satisfies $|H(\bx)| \le 1/2$, and the mean value theorem gives
\[ 
(M-m)^{-1} = (M-m)^{-1} (h_1(M) - h_1(m)) \ll M^2 \ll P^2. 
\]
Thus we have \eqref{unif}, completing the proof of Lemma \ref{major5}.
\end{proof}

By Lemma \ref{major5}, symmetry and H\"older's inequality, it remains to show that
\begin{equation} \label{rest5}
\int_{\fm \cup \ft} |g_1(\alp) g_2(\alp) g_3(\alp)^5 g_8(\alp)^2| K(\alp) \d \alp = o(P^{28/5}).
\end{equation}
We can establish \eqref{rest5} in the same way as \eqref{rest1}, using Lemma \ref{fourth2} instead of Lemma \ref{fourth1}. Indeed, an inspection of that argument shows that the only detail requiring attention is the analogue of \eqref{ClassicalMinor}, which in the present setting becomes
\[ 
\int_\fn |g_i(\alp)^{7-\gam} g_8(\alp)^2| K(\alp) \d \alp  
\ll P^{3(5-\gam)/4+9/5+2\eps}= o(P^{28/5-\gam}).
\]
This completes the proof of Theorem \ref{generalisation}.

\section{Sums of six shifted cubes are dense on the reals}
\label{sixvars}

In this section we prove Theorem \ref{six}. We will need \cite[Theorem 1.4]{MM2011}, so we begin by reviewing the relevant definitions and theory. With future applications in mind, we will be fairly general here. We begin by generalising Definition \ref{irr}. Let $n$ be a positive integer.

\begin{defn} \label{GenIrr} Let $H \in \bR[y_1, \ldots, y_n]$. The polynomial $H$ satisfies the \emph{irrationality condition} if $H(\by) - H(\bzero)$ is not a multiple of a rational polynomial.
\end{defn}

For quadratic forms $Q \in \bR[y_1, \ldots, y_n]$, and for $\bxi \in \bR^n$, define the quadratic polynomial $Q_{\bxi}$ by
\[ Q_{\bxi} (\by) = Q(\by + \bxi). \]
The following is a direct consequence of \cite[Theorem 1.4]{MM2011}.

\begin{thm} \label{MM0} Let $Q$ be a nondegenerate, indefinite quadratic form in $n \ge 3$ variables, and let $\bxi \in \bR^n$. Assume that $Q_\bxi$ satisfies the irrationality condition. Then $Q_\bxi(\bZ^n)$ is dense on $\bR$.
\end{thm}

This implies a similar result for general quadratic polynomials, as we now explain. Let $q \in \bR[y_1, \ldots, y_n]$ be a quadratic polynomial, given by
\[ q(\by) = \sum_{i,j \le n} a_{ij} y_iy_j + \sum_{i \le n} b_i y_i + q(\bzero), \]
where $a_{ij}, b_i \in \bR$ and $a_{ij} = a_{ji}$ ($1 \le i, j \le n$). The \emph{homogeneous part} of $q$ is
\[ Q(\by) = \sum_{i,j \le n} a_{ij} y_iy_j. \]
Assume that $Q$ is nondegenerate, so that $A = (a_{ij})$ is invertible, and put $\bxi = \frac12 A^{-1} \bb$. Then
\begin{align*}
 Q(\by + \bxi) &= \sum_{i,j \le n} a_{ij} (y_i + \xi_i) (y_j + \xi_j) = Q(\by) + (2A \bxi) \cdot \by + Q(\bxi) \\
&= Q(\by) + \bb \cdot \by + Q(\bxi) = q(\by) + Q(\bxi) - q(\bzero).
\end{align*}
In particular, there exists $\bxi \in \bR^n$ such that $q(\by) - Q(\by+ \bxi)$ is constant. We thus arrive at the following corollary of Theorem \ref{MM0}.

\begin{cor} \label{MM} Let $n \ge 3$ be an integer, and let $q \in \bR[y_1, \ldots, y_n]$ be a quadratic polynomial satisfying the irrationality condition. Assume further that the homogeneous part of $q$ is nondegenerate and indefinite. Then $q(\bZ^n)$ is dense on $\bR$.
\end{cor}

Now we are ready to prove Theorem \ref{six}. By fixing the variables $x_7, \ldots, x_s$ if necessary, we may evidently assume that $s = 6$, and that
\begin{equation} \label{wlog2}
0 < \mu_1, \ldots, \mu_6 \le 1.
\end{equation} 
Let $y_1, y_2, y_3 \in \bZ$, and let $a$ equal 3 or 4. Put $x_1 = a+y_1$, $x_4 = -y_1$, $x_2 = y_2$, $x_5 = -y_2$, $x_3 = y_3$, and $x_6 = -y_3$. Now $ F(\bx) = f(a) + 3 \bv \cdot \bc(a)$, where $f(a)$ depends only on $a$ (and $\bmu$), 
\[ \bv = (y_1^2, y_1, y_2^2, y_2, y_3^2, y_3) \]
and
\[ \bc(a) = \begin{pmatrix}
a-  \mu_1 -  \mu_4 \\
(a-\mu_1)^2 - \mu_4^2 \\
-\mu_2 - \mu_5 \\
 \mu_2^2 - \mu_5^2 \\
-\mu_3 - \mu_6 \\
 \mu_3^2 - \mu_6^2
\end{pmatrix}.
\]

In particular, if we choose $a=3$ or $a=4$, then $f(a) + 3\bv \cdot \bc(a)$ is a quadratic polynomial in $\by$, whose homogeneous part is nondegenerate and indefinite. Note that we used \eqref{wlog2} to ensure this. By Corollary \ref{MM}, it remains to show that there exists $a \in \{3,4\}$ such that the entries of $\bc(a)$ are not all in rational ratio. 

Suppose
\[ \frac{ (3-\mu_1)^2 - \mu_4^2}{3-  \mu_1 -  \mu_4 } = 3- \mu_1 + \mu_4 \]
is rational. Then $\mu_1 + \mu_4 \notin \bQ$, since $\mu_1 \notin \bQ$. Now
\[ \frac{4- \mu_1 - \mu_4}{3- \mu_1 - \mu_4} \notin \bQ, \]
so at least one of $(3- \mu_1 - \mu_4) / (-\mu_2 - \mu_5)$ and $(4- \mu_1 - \mu_4) / (-\mu_2 - \mu_5)$ is irrational. This completes the proof of Theorem \ref{six}.

\bibliographystyle{amsbracket}
\providecommand{\bysame}{\leavevmode\hbox to3em{\hrulefill}\thinspace}

\end{document}